\documentclass[12pt]{amsart}
\usepackage[margin=0.6in]{geometry}

\usepackage[english]{babel}
\usepackage[alphabetic]{amsrefs}
\usepackage{amsmath,amssymb,amsfonts,amsthm,enumerate}
\usepackage{url}
\usepackage{graphicx}
\usepackage{wrapfig}

\usepackage[cal=boondox]{mathalfa}

\numberwithin{equation}{section}

\newtheorem{theorem}{Theorem}
\newtheorem{lemma}[theorem]{Lemma}

\newtheorem{corollary}[theorem]{Corollary}

\newcommand{\R}{\mathbb{R}}

\newcommand{\N}{\mathbb{N}}

\renewcommand{\varpi}{\omega}

\renewcommand{\le}{\leqslant}

\renewcommand{\ge}{\geqslant}

\renewcommand{\epsilon}{\varepsilon}

\newcommand{\const}{\,{\rm const}\,}

\begin{document}

\title{Flatness results for nonlocal minimal cones and subgraphs}\thanks{Supported by
the Australian Research Council 
Discovery Project DP170104880 ``N.E.W. --
Nonlocal Equations at Work''. It is a pleasure to thank Serena Dipierro for very interesting
discussions.}

\author{Alberto Farina}

\address{{\em Alberto Farina:} LAMFA,
CNRS UMR 7352,
Facult\'e des Sciences,
Universit\'e de Picardie Jules Verne,
33 rue Saint Leu,
80039 Amiens CEDEX 1, France.
Email: {\tt alberto.farina@u-picardie.fr}}

\author{Enrico Valdinoci}

\address{{\em Enrico Valdinoci:} School of Mathematics and Statistics,
University of Melbourne, 813 
Swanston St, Parkville VIC 3010, Australia,
and Dipartimento di Matematica, Universit\`a degli studi di Milano,
Via Saldini 50, 20133 Milan, Italy, and
Istituto di Matematica Applicata e Tecnologie Informatiche,
Via Ferrata 1, 27100 Pavia, Italy. Email: {\tt enrico@mat.uniroma3.it}}


\begin{abstract}
We show that
nonlocal minimal cones which are non-singular subgraphs
outside the origin are necessarily halfspaces.

The proof is based on classical ideas of~\cite{DG1}
and on the computation of the linearized nonlocal mean curvature operator,
which is proved to satisfy a suitable maximum principle.

With this, we obtain new, and somehow simpler, proofs of the Bernstein-type
results for nonlocal minimal surfaces which have been recently established in~\cite{FV}.
In addition, we establish a new nonlocal Bernstein-Moser-type result
which classifies Lipschitz
nonlocal minimal subgraphs outside a ball.
\end{abstract}

\maketitle

\section{Introduction}

Recently, a Bernstein-type problem for nonlocal minimal surfaces has been settled in~\cite{FV}.
The two main results of~\cite{FV} consist in:
\begin{itemize}
\item first, a ``Lipschitz implies $C^\infty$'' regularity result for nonlocal minimal surfaces,
\item then, a ``no singular cones implies Bernstein Theorem in one more dimension''.
\end{itemize}
The precise statements of these results will be explicitly recalled later on, in Theorems~\ref{FV1.1}
and~\ref{FV1.2}. {F}rom these results, one obtains also interesting byproducts, such as
the fact that $s$-minimal subgraphs are necessarily flat if the ambient space has
dimension less than or equal to~$3$, or less than or equal to~$8$
under the additional assumption that the fractional parameter is large enough.\medskip

The goal of this paper is to 
establish that {\em nonlocal minimal cones with the structure of a non-singular
subgraph outside the origin
are necessarily flat}. Interestingly, this result holds {\em in any dimension}
and provides also a nonlocal version of the
Bernstein Theorem for classical minimal surfaces
proved by J.~Moser in~\cite{MOS}.
As a byproduct, we also obtain new results for~$s$-minimal subgraphs outside a ball.
Furthermore,
the approach of this paper presents an alternative, and somehow simpler, proof
of the results in~\cite{FV}.
\medskip

The mathematical setting in which we work is that introduced in~\cite{CRS},
that we now recall.
We consider an ambient space of dimension~$N:=n+1\in\N$.
For any disjoint (measurable) subsets $X$ and~$Y$ of~$\R^N$ and any~$s\in(0,1)$,
we consider the $s$-interaction of~$X$ and~$Y$, defined by
$$ {\mathcal{I}}_s(X,Y):=\iint_{X\times Y}\frac{dx\,dy}{|x-y|^{N+s}}.$$
Given a (say, bounded and Lipschitz) set~$\Omega\subset\R^N$,
and~$E\subseteq\R^N$, one defines the $s$-perimeter of~$E$ in~$\Omega$
as the sum of all the interactions of~$E$ and~$E^c:=\R^n\setminus\Omega$
to which the domain~$\Omega$ contributes, namely
$$ {\rm Per}_s(E,\Omega):=
{\mathcal{I}}_s(E\cap\Omega,\,E^c\cap\Omega)+
{\mathcal{I}}_s(E\cap\Omega,\,E^c\cap\Omega^c)+
{\mathcal{I}}_s(E\cap\Omega^c,\,E^c\cap\Omega).$$
As customary, the superscript ``$c$'' denotes the complementary set.
Then, one says that~$E$ is an~{\em $s$-minimal set in~$\Omega$}
(or that~$\partial E$ is an~{\em $s$-minimal surface in~$\Omega$}) if
it is a local minimizer in~$\Omega$ of the $s$-perimeter functional,
i.e. if~${\rm Per}_s(E,\Omega)<+\infty$ and
$$ {\rm Per}_s(E,\Omega)\le {\rm Per}_s(F,\Omega)$$
for every~$F\subseteq\R^N$ such that~$F\cap\Omega^c=E\cap\Omega^c$.\medskip

We also say that~$E$ is an~{\em $s$-minimal set if it is
an~$s$-minimal set} in~$B_R$ for all~$R>0$.\medskip

The study of $s$-minimal surfaces is extremely
challenging and fascinating,
and their complete regularity theory is one of the most important open problems
in the topic of fractional analysis. Till now, it is known that $s$-minimal surfaces
are $C^\infty$ in the interior of the reference domain
when the dimension~$N$ of the ambient space is less than or equal than~$3$
(see~\cite{SV}) and when the dimension~$N$ of the ambient space is less than or equal than~$8$
as long as the fractional exponent~$s$ is sufficiently close to~$1$ (see~\cite{CV}
and also~\cite{barrios}).\medskip

The boundary regularity is somehow a different story with respect to the interior case,
since $s$-minimal surfaces have the tendency to stick at the boundary of the domain and
to detach in a $C^{1,\frac{1+s}2}$ fashion\footnote{As a ``philosophical remark'',
let us mention that the exponent~$\frac{1+s}{2}$
is somehow consistent with the kernel in computations like that in~\eqref{E:1}
in this paper, in which the kernel is of the form~$n+2S$, with~$S:=\frac{1+s}{2}$.
These types of operators are sort of nonlinear $S$-Beltrami Laplacian
along manifolds of dimension~$n$ embedded into~$\R^N=\R^{n+1}$.} from it, see~\cite{dan, DSV17}.\medskip

Nonlocal perimeters and nonlocal minimal surfaces have also a number
of applications, and they naturally
arise for instance as interfaces of long-range phase coexistence models
(see~\cite{SV2}) and models for cellular automata (see~\cite{takis}).
An intense research activity has been done concerning nonlocal
isoperimetric problems (see e.g.~\cites{GR, FSISO, FMM, I5})
constant nonlocal mean curvature surfaces
(see e.g.~\cite{CFW, ciraolo} and also~\cite{NA}) and nonlocal
geometric flows (see~\cite{I, CMP15, saez, sinestrari}), and the topic is rich
of very challenging and important open questions, with many links to other subjects,
see e.g.~\cites{CIME, survey} for recent surveys.\medskip

We say that a set~$E\subseteq\R^N$ is a {\em cone} (with respect to the origin)
if for every~$p\in E$ we have that~$tp\in E$
for any~$t\ge0$.\medskip

We say that a set~$E\subseteq\R^N$ is an {\em $s$-minimal cone} 
if it is an $s$-minimal set and it is a cone.\medskip

It is also interesting to consider the case in which sets have the structure of subgraphs
(say, for definiteness, with respect to the last coordinate).
That is, we say that~$E\subseteq\R^N=\R^{n+1}$ is a {\em subgraph} if there exists
a (measurable) function~$u:\R^n\to\R$ such that
\begin{equation}\label{GRAPH} E=\{ x_{n+1}<u(x_1,\dots,x_n)\}.\end{equation}
We say that~$E$ is an {\em $s$-minimal subgraph} if it is an~$s$-minimal set and a subgraph.\medskip

We remark that the notion of subgraph fits well into the nonlocal minimal surface setting,
since if one considers cylindrical domains~$\Omega=\Omega_o\times\R\subset\R^{n+1}$
with~$\Omega_o\subset\R^n$ and prescribes the data of the set outside~$\Omega$
to be a subgraph, then the $s$-minimal sets in~$\Omega$ possess a subgraph structure,
though with possible discontinuities along the boundary of the cylinder
(see~\cite{DSV16}, and this boundary discontinuity is an important difference
with respect to the classical minimal surfaces).\medskip

In this context, the main result of this paper is the following:

\begin{theorem}\label{ADG}
Let~$ n\ge 1$ and~$ E$ be an $s$-minimal cone in~$\R^{n+1}$.
Assume that~$E$ is a subgraph, as in~\eqref{GRAPH}.

Suppose that
\begin{equation}\label{RR}
{\mbox{$u$ is locally Lipschitz continuous in~$\R^n\setminus\{0\}$.}}\end{equation} Then $E$ is a halfspace.
\end{theorem}

This result is the nonlocal analogue of that on page~79 of~\cite{DG1}
for the minimizers of the classical perimeter functional.

Theorem~\ref{ADG} can be also recasted for
conical subgraphs with zero fractional
mean curvature. For this, given~$E\subset\R^N$
and~$x\in\partial E$, with~$\partial E$ of class~$C^{1,\alpha}$ with~$\alpha>s$ at~$x$,
we define the $s$-mean curvature of~$E$ at~$x$ as
$$ {\mathcal{H}}^E_s(x):=\int_{\R^n}\frac{\chi_{E^c}(y)-\chi_E(y)}{|x-y|^{n+s}},
$$
in the principal value sense.
See~\cites{CRS, AB, LOMB, I5} for the basic properties of~${\mathcal{H}}^E_s$
and for the fact that it can be considered as the first variation of~${\rm Per}_s$.
In this setting, one can restate Theorem~\ref{ADG}
by saying that {\em if a set
is a subgraph and a cone, that is smooth outside the origin,
and with~${\mathcal{H}}^E_s=0$ on~$\partial E\setminus\{0\}$,
then it must be a halfspace.
}
The structural assumptions of such statement
are somehow sharp, since the cones discussed in Theorem~3 of~\cite{DDW}
are smooth outside the origin and have
zero fractional mean curvature in any dimension, without being halfspaces (but
such cones are not subgraphs).
\medskip

As an immediate consequence of Theorem~\ref{ADG}, we obtain 
a new proof of the following two results, which have been recently proved, with
other methods, in~\cite{FV}:

\begin{theorem}[Theorem 1.1 in~\cite{FV}]\label{FV1.1}
Let $n\ge 1$ and let~$ E$ be an $s$-minimal set in~$ B _1\subset\R^{ n+1 }$. Suppose
that~$\partial E\cap B_1$ is locally Lipschitz. Then~$\partial E\cap B_1$ is~$ C^\infty$.
\end{theorem}

\begin{theorem}[Theorem 1.2 in~\cite{FV}]\label{FV1.2}
Let~$
E=\{ x_{n+1}<u(x_1,\dots,x_n)\}$
be an $s$-minimal subgraph, and
assume that there are no singular $s$-minimal cones in dimension~$ n$ 
(that is, if~${\mathcal{ C}}\subset\R^n$ is a non-empty
$s$-minimal cone, then~${\mathcal{ C}}$
is a halfspace). Then~$ u$ is an affine function (thus $E$ is a halfspace).
\end{theorem}

The proofs of Theorems~\ref{FV1.1} and~\ref{FV1.2}
are based on the classification of blow-up and blow-down cones respectively (see
e.g. Section~3 in~\cite{FV} or Section~\ref{CONISUR1}
here for a detailed discussion on blow-up and blow-down cones).
In~\cite{FV} such classification of limit cones
is obtained by geometric methods: roughly speaking,
the idea in~\cite{FV} is to use the cone itself as a barrier, modifying it in a given region
by adding a ``small bump'' to the original cone 
and sliding this modified surface to a first contact point with the original one
(on the one hand, the effect of this bump should change
only by little the nonlocal mean curvature at a smooth point, on the other hand
such nonlocal mean curvature is influenced from far-away points by an order one contribution
in a non-flat picture, and these two observations provide
a contradiction).\medskip

Very roughly speaking,
the idea for the argument of~\cite{FV} is quite geometric
and sketched in Figure~\ref{WIK2} (after a dilation, for
instance, the subgraph reduces to a cone,
which is touched from below by the dashed surface,
which should have a small nonlocal mean curvature
if the additional bump is small, in contradiction
with the mass produced in the grey region).\medskip

\begin{figure}
    \centering
    \includegraphics[width=11cm]{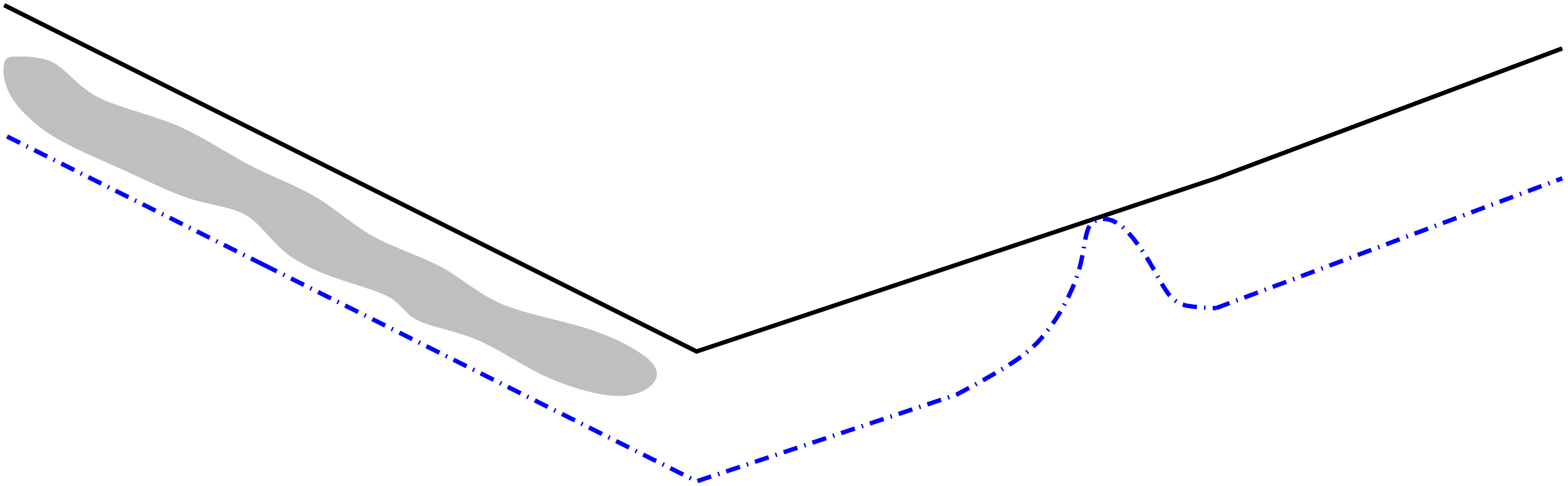}
    \caption{\it {{A geometric idea in \cite{FV}.}}}
    \label{WIK2}
\end{figure}

In this paper instead
we obtain Theorems~\ref{FV1.1} and~\ref{FV1.2} directly from
Theorem~\ref{ADG}, which classifies the blow-up and blow-down cones
without any other additional geometric argument
and only relying on analytical methods, such as linearization
and maximum principle.\medskip

It is also interesting to point out that our results can also comprise
the more general settings of subgraphs outside a ball. Namely, we say that~$E
\subseteq\R^N=\R^{n+1}$ is a {\em subgraph outside a ball} if there exists
a ball~$B\subset\R^n$ and
a function~$u:\R^n\to\R$ such that
\begin{equation}\label{GRAPH:BB} 
E \setminus (B\times\R)=\{ x_{n+1}<u(x_1,\dots,x_n)\}\setminus (B\times\R)
.\end{equation}
We say that~$E$ is an {\em $s$-minimal subgraph
outside a ball} if it is an~$s$-minimal set and a subgraph outside a ball.\medskip

The setting that we provide in Theorem~\ref{ADG} is
strong enough to establish a classification result
for Lipschitz
$s$-minimal subgraphs outside a ball, which goes
as follows:

\begin{theorem}\label{MOSBE}
Let~$ n\ge 1$ and~$ E$ be an $s$-minimal subgraph outside an open ball~$B$, as in~\eqref{GRAPH:BB}.
Suppose also that
\begin{equation}\label{REGMOS}
{\mbox{$u$ is
globally Lipschitz continuous in~$\R^n\setminus B$.}}\end{equation}
Then $E$ is a halfspace.
\end{theorem}

Theorem~\ref{MOSBE} can be seen as the nonlocal counterpart
of the results for classical minimal surfaces that were obtained
in Section~7 of~\cite{MOS} by means of Harnack-type inequalities.
As far as we know, Theorem~\ref{MOSBE} is new even for $s$-minimal subgraphs as in~\eqref{GRAPH}.
\medskip

We also point out that a strengthening of Theorem~\ref{FV1.1}
holds true, where one can drop the Lipschitz assumption
if no singular cones exist in one dimension less, according to the following statement:

\begin{theorem}\label{noC}
Let $n\ge 1$ and let~$ E$ be an $s$-minimal subgraph in~$ B _1\subset\R^{ n+1 }$. 
Assume also that there are no singular $s$-minimal cones in dimension~$ n$.
Then~$\partial E\cap B_1$ is~$ C^\infty$.
\end{theorem}

{F}rom Theorem~\ref{noC} here and Theorem~1
in~\cite{SV}, it plainly follows that:

\begin{corollary}\label{C6}
Let $n\in\{ 1,\,2\}$ and let~$ E$ be an $s$-minimal subgraph in~$ B _1\subset\R^{ n+1 }$. 
Then~$\partial E\cap B_1$ is~$ C^\infty$.
\end{corollary}

Similarly, from Theorem~\ref{noC} here and Theorem~2 in~\cite{CV},
we obtain:

\begin{corollary}\label{C7}
Let $n\in\{ 1,\dots, 7\}$.
Then there exists~$\epsilon_n\in(0,1)$ such that if~$s\in (1-\epsilon_n,\,1)$ 
and~$ E$ is an $s$-minimal subgraph in~$ B _1\subset\R^{ n+1 }$, it holds that~$\partial E\cap B_1$ is~$ C^\infty$.
\end{corollary}

Corollaries~\ref{C6} and~\ref{C7} can be seen as a positive answer when~$n\in\{ 1,\,2\}$
(or when~$n\in\{ 1,\dots, 7\}$ and~$s$ is close enough to~$1$) to the Open Problem 
in Section~7.3 of~\cite{CIME}.\medskip

The rest of this paper is mostly devoted to the proof of Theorem~\ref{ADG},
which in turn will imply Theorems~\ref{FV1.1} and~\ref{FV1.2} right away.
We then give the proofs of
Theorems~\ref{MOSBE}
and~\ref{noC}, by exploiting Theorem~\ref{ADG}. Before presenting the main arguments
of the proof of Theorem~\ref{ADG}, we present some basic facts about blow-up and blow-down cones.

The paper ends with an appendix collecting some ancillary (but nontrivial) estimates
to check that some integrals are well-defined and can be differentiated.

\section{An overview on blow-up and blow-down cones for $s$-minimal sets}\label{CONISUR1}

Given an $s$-minimal set~$E$, with~$0\in\partial E$, for any~$r>0$,
we consider the family of sets~$ E_r:=E/r$.
By density estimates and monotonicity formulas, we know that~$E_r$ converges
up to subsequences to some~$E_0$ as~$r\searrow0$
and to some~$E_\infty$ as~$r\nearrow+\infty$. In addition,
the sets~$E_0$ and~$E_\infty$ are $s$-minimal
cones, see Theorem~9.2 in~\cite{CRS}.  \medskip

The cone~$E_0$, which is called in jargon
``blow-up cone'', corresponds to the action of ``looking the picture close to the origin by
performing a zoom-in''
and its flatness is equivalent to the regularity of the original set~$E$ in a neighborhood
of the origin.
Indeed, from Theorem~6.1 in~\cite{CRS}
and Theorem~5 in~\cite{barrios}, we have that 
\begin{equation}\label{up}
{\mbox{if $E_0$ is a halfspace, then $\partial E\cap B_\rho$ is $C^\infty$,}}  
\end{equation}
for some~$\rho>0$.\medskip

The cone~$E_\infty$ is called in jargon
``blow-down cone'' and corresponds to the action of ``looking the picture from far, by
making a zoom-out''
and its flatness is equivalent to the flatness of the original set~$E$.
Indeed, we have that 
\begin{equation}\label{down}
{\mbox{if $E_\infty$ is a halfspace, then $E=E_\infty$, and so~$E$ is a halfspace.}}  
\end{equation}
See e.g. Lemma 3.1 in~\cite{FV}.
\medskip

Moreover, we recall the following dimensional 
reduction
of Theorem~10.3 in~\cite{CRS} (see also Theorem~5.33
in~\cite{LOMB}):

\begin{lemma}\label{doppio cono}
Let~$E$ be an $s$-minimal cone in~$\R^{n+1}$.
Let~$p\in\partial E$,
with~$p\ne0$.

Let~$F:=E-p$ and let~$F_0$ be
the blow-up cone for~$F$.

Then~$F_0$
can be written up to a rotation as the Cartesian product~$\tilde F\times \R$,
where~$\tilde F$ is an $s$-minimal cone in~$\R^n$.

Also, if~$\partial E$ is not~$C^\infty$ at~$p$, then~$\tilde F$
is singular at the origin.\end{lemma}

{F}rom Lemma~\ref{doppio cono}, we obtain that if no singular $s$-minimal
cones exist in one dimension
less, then the origin is the only possible singularity of $s$-minimal cones
(i.e. if no singular $s$-minimal cones exist in~$\R^n$, then the
$s$-minimal cones in~$\R^{n+1}$ are either halfspaces or singular only at the origin):

\begin{corollary}\label{doppio cor}
Let~$E$ be an $s$-minimal cone in~$\R^{n+1}$.
Assume that there are no singular $s$-minimal cones in dimension~$ n$.
Then, $\partial E$ is $C^\infty$ outside the origin.
\end{corollary}

\begin{proof}
Suppose, by contradiction, that~$\partial E$ is not~$C^\infty$ at some point~$p\in\partial E$,
with~$p\ne0$. Then Lemma~\ref{doppio cono}
produces a singular minimal cone~$\tilde F$ in~$\R^n$, which
is in contradiction with our assumption.\end{proof}

The dimension reduction of Lemma~\ref{doppio cono}
can be strengthen in case of Lipschitz subgraphs. To this end,
we first give a detailed blow-up argument when a Lipschitz
assumption is taken:

\begin{lemma}\label{0p045:A}
Let~$E$ be an $s$-minimal cone in~$\R^{n+1}$.
Let~$\kappa\in\R$ and~$p=(0,\dots,0,1,\kappa)\in\partial E$
and suppose that~$\partial E$ is a Lipschitz subgraph in a neighborhood of~$p$.

Let~$F:=E-p$ and let~$F_0$ be
the blow-up cone for~$F$.
Then, $F_0$ is the subgraph of some Lipschitz function~$v_0:\R^n\to\R$.
Furthermore, there exists a Lipschitz function~$ v^\star:\R^{n-1}\to\R$ such that
\begin{equation} \label{PA2L}
v_0(\hat x,x_n)= v^\star(\hat x)+\kappa x_n,\end{equation}
for any~$(\hat x,x_n)\in\R^{n-1}\times\R$.

In addition, for any~$t\ge0$ and any~$\hat x\in\R^{n-1}$,
\begin{equation}\label{0weiuu2y}
v^\star(t\hat x)=tv^\star(\hat x).
\end{equation}

Furthermore, if~$\left(-\frac\pi2,\frac\pi2\right)\ni\theta:=\arctan\kappa$
and we consider the rotation given by
\begin{equation}\label{ROROT}\begin{split}&
\R^{n-1}\times\R\times\R\ni
(\hat y, y_n,y_{n+1}):={\mathcal{R}}_\theta(\hat x, x_n,x_{n+1}),\\&
\qquad{\mbox{ with}}\qquad
\hat y:=\hat x,\qquad
\left(\begin{matrix} y_n\\ y_{n+1}
\end{matrix}\right)
:=
\left(\begin{matrix} \cos\theta & \sin\theta
\\ -\sin\theta &
\cos\theta
\end{matrix}\right)
\left(\begin{matrix} x_n \\ x_{n+1}
\end{matrix}\right),\end{split}\end{equation}
we have that~${\mathcal{R}}_\theta(F_0)=F^\star\times\R$, where~$F^\star$
is an~$s$-minimal cone in~$\R^n$ 
(with variables~$(\hat y,y_{n+1})$)
and it is the subgraph~$
\{ y_{n+1}<\cos\theta v^\star(\hat y)\}$.
\end{lemma}

\begin{proof} We let~$e_n:=(0,\dots,0,1)\in\R^n$ and~$e_{n+1}:=(0,\dots,0,0,1)\in\R^{n+1}$.
In this way, we can write~$p=(e_n,0)+\kappa e_{n+1}$.
We let~$u$ be the subgraph describing~$E$ near~$p$ and~$v(x):=u(x+ e_n)-\kappa$.
The function~$v$ describes the subgraph of~$F$ near the origin.
Let~$M>0$ be the Lipschitz constant of~$u$ in~$B_\varrho(e_n)$, for some~$\varrho>0$.
Then the Lipschitz constant of~$v$ in~$B_\varrho$ is bounded by~$M$.

We also define~$v_r(x):=\frac{v(rx)}{r}$. By construction, $v(0)=u(e_n)-\kappa=0$
and so~$v_r(0)=0$. Also, the Lipschitz constant of~$v$ in~$B_{\varrho/r}$ is bounded by~$M$.
Consequently, from the Arzel\`a-Ascoli Theorem,
up to subsequence,
as~$r\searrow0$,~$v_r$ converges locally uniformly to some Lipschitz function~$v_0$.

Now, since~$E$ is a cone, for any~$\hat y\in\R^{n-1}$ and~$t\ge0$,
we have that
$$ u(t \hat y, t)=t u(\hat y,1).$$
Accordingly, fixed~$\hat x\in\R^{n-1}$ and~$\tau\ge-\frac1r$, taking~$t:=r\tau+1\ge0$
and~$\hat y:=\frac{r \hat x}{r\tau+1}=\frac{r\hat x}{t}$, we see that
\begin{eqnarray*}
&& v(r\hat x, r\tau)=u(r\hat x, r\tau+1)-\kappa=
u\big( t\hat y, t\big)-\kappa=t u(\hat y,1)-\kappa=
r\tau u(\hat y,1)+u(\hat y,1)-\kappa\\
&&\qquad\qquad\qquad =r\tau u\left(\frac{r \hat x}{r\tau+1},1\right)+u\left(\frac{r \hat x}{r\tau+1},1\right)-\kappa
\end{eqnarray*} 
and thus
\begin{equation}\label{G1} 
v_r(\hat x,\tau)=
\tau u\left(\frac{r \hat x}{r\tau+1},1\right)+
\frac{
u\left(\frac{r \hat x}{r\tau+1},1\right)-u({r \hat x},1)}{r}+
\frac{
u({r \hat x},1)-\kappa}{r}.
\end{equation}
So, we now fix~$\tau\in\R$ and~$\hat x\in\R^{n-1}$. We take~$r>0$ so small that~$\tau\ge-\frac1r$.
Also, for small~$r$, we have that~$\left(\frac{r \hat x}{r\tau+1},1\right)$
and~$({r \hat x},1)$ belong to~$ B_{\varrho/2}(e_n)$
and therefore
\begin{equation}\label{G2} 
\left|\frac{
u\left(\frac{r \hat x}{r\tau+1},1\right)-u({r \hat x},1)}{r}\right|\le
\frac{M\,
\left|\frac{r \hat x}{r\tau+1}-{r \hat x}\right|}{r}\le M\,|\hat x|\,|\tau|\,r,\quad{\mbox{ which
is infinitesimal as }}r\searrow0.
\end{equation}
We also set~$w_r(\hat x):=\frac{
u({r \hat x},1)-\kappa}{r}$. Notice that~$w_r(0)=\frac{
u(0,1)-\kappa}{r}=0$, and~$w_r$ has Lipschitz constant locally bounded by~$M$ when~$r$ is small.
Thus, up to a subsequence, we suppose that~$w_r$ converges locally uniformly to
a Lipschitz function~$v^\star$.

Using this information and~\eqref{G2}, we can pass to the limit in~\eqref{G1}
and conclude that
\begin{eqnarray*} v_0(\hat x,\tau)&=&\lim_{r\searrow0}v_r(\hat x,\tau)\\
&=& \lim_{r\searrow0}
\tau u\left(\frac{r \hat x}{r\tau+1},1\right)+
\frac{
u\left(\frac{r \hat x}{r\tau+1},1\right)-u({r \hat x},1)}{r}+
w_r(\hat x)\\&=&
\tau u(0,1)+
v^\star(\hat x)\\
&=& \kappa\tau+v^\star(\hat x)
.\end{eqnarray*}
This establishes~\eqref{PA2L}, as desired.

Now, from~\eqref{PA2L} and the fact that~$F_0$ is a cone, we have that
\begin{eqnarray*}
v^\star(t\hat x)+\kappa t x_n
=
v_0(t\hat x,tx_n)=tv_0(\hat x,x_n)=t\big( v^\star(\hat x)+\kappa x_n\big)=
tv^\star(\hat x)+\kappa t x_n,
\end{eqnarray*}
from which~\eqref{0weiuu2y} easily follows.

Now, we consider the rotation in~\eqref{ROROT},
and we perform a geometric argument to complete the proof of Lemma~\ref{0p045:A},
as described by Figure~\ref{DESS}.

\begin{figure}
    \centering
    \includegraphics[width=11cm]{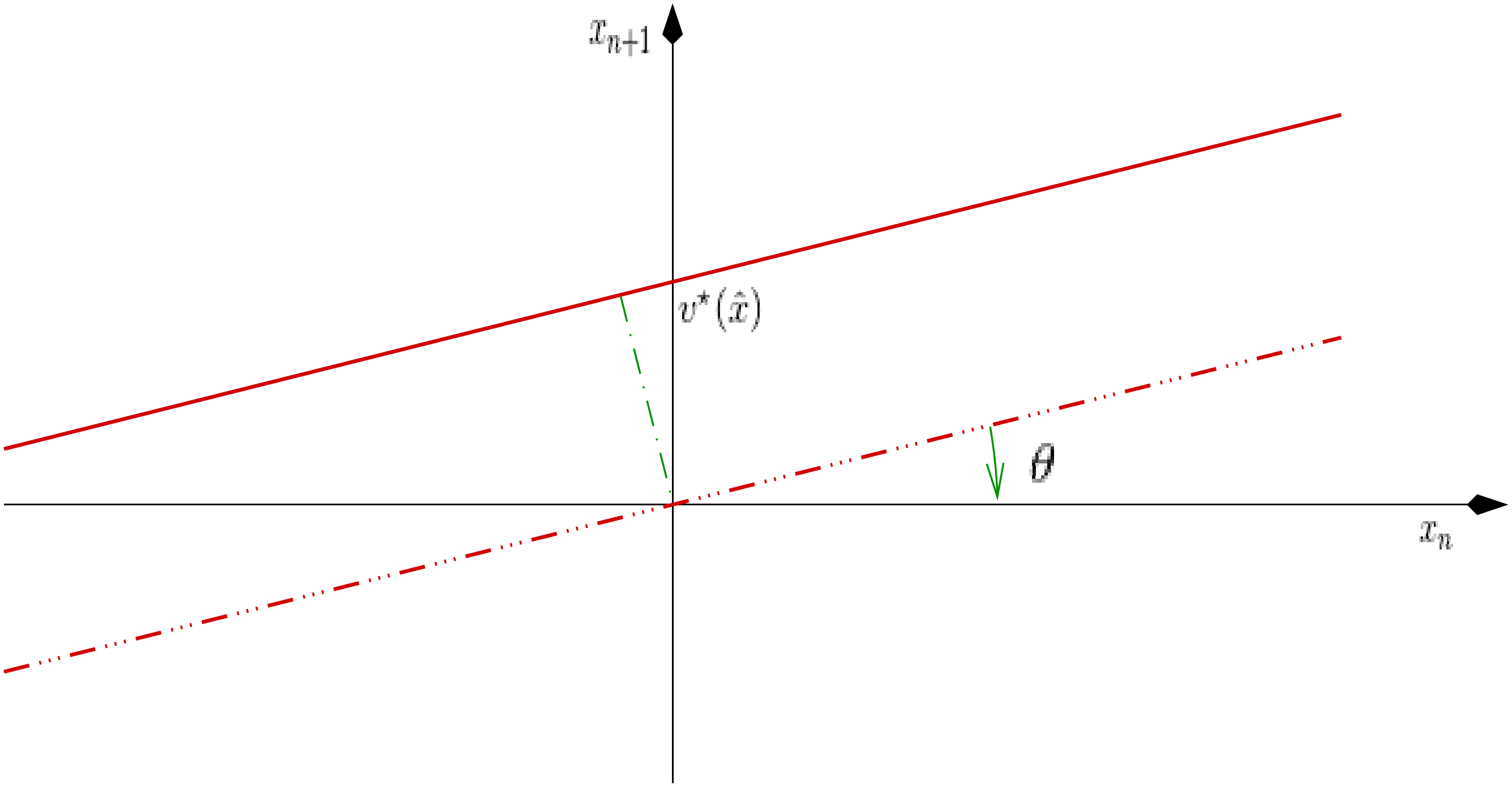}
    \caption{\it {The effect of the rotation~${\mathcal{R}}_\theta$
on the plane~$(x_n,x_{n+1})$ for a fixed~$\hat x\in\R^{n-1}$.}}
    \label{DESS}
\end{figure}

The analytic details go as follows.
We substitute~$\hat x=\hat y$, $x_{n}=\cos\theta\, y_n-\sin\theta\, y_{n+1}$
and~$x_{n+1}=\sin\theta \,y_n+\cos\theta\, y_{n+1}$. In this way, we have that
\begin{eqnarray*}
{\mathcal{R}}_\theta(F_0) 
&=& {\mathcal{R}}_\theta\Big( \big\{ x_{n+1}< v_0(\hat x,x_n)\big\}\Big)
\\
&=& {\mathcal{R}}_\theta\Big( \big\{ x_{n+1}< v^\star(\hat x)+
\tan\theta\,x_n\big\}\Big)
\\
&=& \big\{ \sin\theta \,y_n+\cos\theta\, y_{n+1}< 
v^\star(\hat y)+\tan\theta\,( \cos\theta\, y_n-\sin\theta\, y_{n+1})\big\}\\
&=& \big\{ (\cos\theta
+
\tan\theta\,\sin\theta) y_{n+1}
< 
v^\star(\hat y)\big\}
\\ &=& \big\{  y_{n+1}
< \cos\theta\,
v^\star(\hat y)\big\}\\
&=& \{(\hat y,y_n,y_{n+1}) {\mbox{ with }} 
y_{n+1}
< \cos\theta\,
v^\star(\hat y) {\mbox{ and }}y_n\in\R\big\}.
\end{eqnarray*}
This gives the desired expression for~$F^\star$.
Also, since~$F_0$ is $s$-minimal in~$\R^{n+1}$, so is~${\mathcal{R}}_\theta(F_0)=
F^\star\times\R$ and then~$F^\star$ is $s$-minimal in~$\R^n$
by dimension reduction
(see Theorem~10.3 in~\cite{CRS} or Theorem~5.33
in~\cite{LOMB}).
\end{proof}

\section{Proof of Theorem~\ref{ADG}}\label{565441}

The idea of the proof, inspired by the work in~\cite{DG1} (with the very minor
correction in~\cite{DG2}),
is based on two observations:
\begin{itemize}
\item the derivatives of~$u$ should satisfy a linearized equation,
\item since~$E$ is a cone, these derivatives are
positively homogeneous of degree zero, and therefore
their maxima and minima along the unit sphere
must be global maxima and minima.\end{itemize}
Then, if the linearized equation possesses a maximum, or minimum, principle,
the derivatives of~$u$ are necessarily constant, hence~$u$ must be affine
and~$E$ is proved to be a halfspace, as desired.\medskip

In our setting, the goal is thus to write an appropriate equation for~$u$
and for its derivatives and check that a suitable maximum principle
is satisfied for such linearized equation. The main idea for this
is that the fractional mean curvature of~$E$ at any boundary point
must vanish (see Theorem~5.1~\cite{CRS}
or Corollary~4.13 in~\cite{LOMB}). Then, by some computation one can write
such fractional mean curvature directly in term of the function~$u$:
this expression is not completely explicit, since one cannot write ``simple expressions''
of all the integrals involved, nevertheless the formulas that we provide
(which can be useful in other contexts as well) are suitable to carry out detailed analytic arguments.
Indeed, outside the origin the function~$u$ is differentiable
and the fractional mean curvature equation can be linearized.
A careful inspection of the sign of the integrands reveals that this linearized
equation ``charges positive masses'' and so it possess a sort of maximum principle,
which in turn provides the desired result.\medskip

In further details, the analytical ingredients which
effectively implement this general strategy go as follows. 
We write~$E$ as a subgraph, as in~\eqref{GRAPH}.
First, we establish
Theorem~\ref{ADG} under the additional assumption that
\begin{equation}\label{REGA}
u\in C^{3}(\R^n\setminus\{0\}).\end{equation}
Since~$E$ is $s$-minimal, at any point~$X=(x,u(x))\in\R^{n+1}$,
with~$x\in \R^n\setminus\{0\}$, we have that the fractional mean
curvature of~$E$ at~$X$ must vanish (see Theorem~5.1 
in~\cite{CRS} or Corollary~4.13 in~\cite{LOMB}).
We also consider the tangent halfspace at~$X$, namely
$$ L_X:=\{ Y=(y,y_{n+1})\in\R^n\times\R {\mbox{ s.t. }} y_{n+1}<\nabla u(x)\cdot (y-x)+u(x)\}.$$
Of course, the fractional mean
curvature of~$L_X$ at~$X$ must vanish as well.
Notice in addition that
$$ \chi_E-\chi_L=\left\{
\begin{matrix}
1 & {\mbox{ in }}E\cap L^c\\
-1 & {\mbox{ in }}L\cap E^c\\
0 & {\mbox{ otherwise }}
\end{matrix}
\right. $$
and similarly
$$ \chi_{E^c}-\chi_{L^c}=\left\{
\begin{matrix}
1 & {\mbox{ in }}E^c\cap L\\
-1 & {\mbox{ in }}L^c\cap E\\
0 & {\mbox{ otherwise }}
\end{matrix}
\right. $$
These considerations imply that, for any~$X=(x,x_{n+1})=(x,u(x))\in\R^n\times\R$,
with~$x\ne0$,
\begin{equation}\label{E:1}\begin{split}
0\; &= \frac12 \int_{\R^N} \frac{\big(\chi_{E^c}(Y)-\chi_{L^c}(Y)\big)
-\big(\chi_{E}(Y)-\chi_{L}(Y)\big) }{|X-Y|^{N+s}}\,dY\\
&= \int_{E^c\cap L} \frac{dY}{|X-Y|^{n+1+s}}-\int_{E\cap L^c} \frac{dY}{|X-Y|^{n+1+s}}\\
&= \int_{{Y=(y,y_{n+1})\in\R^n\times\R}\atop{u(y)\le
y_{n+1}<\nabla u(x)\cdot (y-x)+u(x)
}} \frac{dY}{|X-Y|^{n+1+s}}-\int_{{Y=(y,y_{n+1})\in\R^n\times\R}\atop{\nabla u(x)\cdot (y-x)+u(x)\le
y_{n+1}<u(y)
}} \frac{dY}{|X-Y|^{n+1+s}} 
\end{split}\end{equation}
Now, fixed~$X=(x,x_{n+1})\in\R^n\times\R$, given two functions~$f$, $g:\R^n\to\R$,
with~$f(x)=g(x)=x_{n+1}$, we 
use the substitutions
$$ \R\ni\tau:=\frac{y_{n+1}-x_{n+1}}{|y-x|} \quad{\mbox{ and }}\quad
\R^n\ni \vartheta:=y-x$$
to make the following calculation:
\begin{equation}\label{E:2}
\begin{split}
\int_{{Y=(y,y_{n+1})\in\R^n\times\R}\atop{f(y)\le
y_{n+1}<g(y)}} \frac{dY}{|X-Y|^{n+1+s}}
\;&= \int_{{Y=(y,y_{n+1})\in\R^n\times\R}\atop{f(y)\le
y_{n+1}<g(y)}} \frac{dY}{|y-x|^{n+1+s}\left( 1+\frac{|y_{n+1}-x_{n+1}|^2}{|y-x|^2}\right)^{\frac{n+1+s}{2}}}\\
&=\int_{{(\vartheta,\tau)\in\R^n\times\R}\atop{
\frac{f(x+\vartheta)-x_{n+1}}{|\vartheta|}\le
\tau<\frac{g(x+\vartheta)-x_{n+1}}{|\vartheta|}}} \frac{d\vartheta\,d\tau}{|\vartheta|^{n+s}
\left( 1+\tau^2\right)^{\frac{n+1+s}{2}}}\\&=
\int_{\R^n}\frac{
F\left( \frac{g(x+\vartheta)-x_{n+1}}{|\vartheta|}\right)
-F\left( \frac{f(x+\vartheta)-x_{n+1}}{|\vartheta|}\right)
}{|\vartheta|^{n+s} }\,d\vartheta\end{split}\end{equation}
where, for any~$r\in\R$, we used the notation\footnote{The function~$F$ was introduced
in formula~(49) of~\cite{barrios}.
See also~\cite{CFW} for other applications of such function
in related contexts.}
\begin{equation}\label{DEFF} 
F(r):=\int_0^r \frac{d\tau}{\left( 1+\tau^2\right)^{\frac{n+1+s}{2}}}.\end{equation}
Making use of~\eqref{E:2} into~\eqref{E:1},
with the choice~$f(y):=u(y)$ and~$g(y):=\nabla u(x)\cdot(y-x)+u(x)$, we conclude that
\begin{equation}\label{PRE:3} {\mathcal{F}}[u](x):=\int_{\R^n}\frac{
F\left( \frac{u(x+\vartheta)-u(x)}{|\vartheta|}\right)
-F\left( \frac{\nabla u(x)\cdot\vartheta}{|\vartheta|}\right)
}{|\vartheta|^{n+s} }\,d\vartheta =0,\end{equation}
for any~$x\in\R^n\setminus\{0\}$.

To avoid integrals in the principal value sense,
it is now convenient to produce a ``more symmetric'' version of~\eqref{PRE:3},
by taking into account the fact that~$F'(r)=\frac{1}{(1+r^2)^{\frac{n+1+s}{2}}}$, which
is an even function,
and so~$F$ is an odd function. Therefore, the change of variable~$\vartheta\mapsto-\vartheta$ gives that
\begin{eqnarray*}
\int_{\R^n}\frac{
F\left( \frac{u(x+\vartheta)-u(x)}{|\vartheta|}\right)
-F\left( \frac{\nabla u(x)\cdot\vartheta}{|\vartheta|}\right)
}{|\vartheta|^{n+s} }\,d\vartheta&=&
\int_{\R^n}\frac{
F\left( \frac{u(x-\vartheta)-u(x)}{|\vartheta|}\right)
-F\left( -\frac{\nabla u(x)\cdot\vartheta}{|\vartheta|}\right)
}{|\vartheta|^{n+s} }\,d\vartheta\\
&=&
\int_{\R^n}\frac{
F\left( \frac{u(x-\vartheta)-u(x)}{|\vartheta|}\right)
+F\left(\frac{\nabla u(x)\cdot\vartheta}{|\vartheta|}\right)
}{|\vartheta|^{n+s} }\,d\vartheta
\end{eqnarray*}
Summing up
this to~\eqref{PRE:3}, we conclude that
\begin{equation}\label{E:3} 2{\mathcal{F}}[u](x):=\int_{\R^n}\frac{
F\left( \frac{u(x+\vartheta)-u(x)}{|\vartheta|}\right)+
F\left( \frac{u(x-\vartheta)-u(x)}{|\vartheta|}\right)}{|\vartheta|^{n+s} }\,d\vartheta =0.\end{equation}
One can check that
\begin{equation}\label{WP1}
{\mbox{the integral in~\eqref{E:3} is
well-posed,}}
\end{equation}
see Appendix~\eqref{AWP1}.
Then, for any~$j\in\{1,\dots,n\}$ we take\footnote{We also observe that, in~\eqref{E:4},
one can also write the short-hand 
expression
\[ {\mathcal{L}}_u [v](x)\,=\,2\int_{\R^n}\frac{
F'\left( \frac{u(x+\vartheta)-u(x)}{|\vartheta|}\right)\,\Big( v(x+\vartheta)-v(x)\Big)
}{|\vartheta|^{n+s+1} }\,d\vartheta,\]
provided that the integral is taken in the principal value sense.} a derivative of~\eqref{E:3}
and we obtain that~$v:=\partial_j u$ satisfies the linearized equation
\begin{equation}\label{E:4}\begin{split} {\mathcal{L}}_u [v](x)\,:=\;&
\int_{\R^n}\frac{
F'\left( \frac{u(x+\vartheta)-u(x)}{|\vartheta|}\right)\,\Big( v(x+\vartheta)-v(x)\Big)+
F'\left( \frac{u(x-\vartheta)-u(x)}{|\vartheta|}\right)\,\Big( v(x-\vartheta)-v(x)\Big)
}{|\vartheta|^{n+s+1} }\,d\vartheta\\ =\;&0.\end{split}\end{equation}
The fact that the derivative in~\eqref{E:3} can be safely taken, thus leading to~\eqref{E:4}
is checked in details in Appendix~\ref{APPP2}.

Now we claim that
\begin{equation}\label{E:5}
{\mbox{$v$ is constant.}}
\end{equation}
The proof of~\eqref{E:5} is based on the maximum principle, applied to the linearized
equation in~\eqref{E:4}. In order to achieve this aim,
we take~$\bar x\in S^{n-1}$ be such that
$$ v(\bar x)=\max_{S^{n-1}} v.$$
Since~$E$ is a cone, it follows that~$u$ is positively homogeneous of degree one,
and so~$v$ is positively homogeneous of degree zero.
As a result, it follows that, for any~$x\in\R^n\setminus\{0\}$,
$$ v(x)=v\left( \frac{x}{|x|}\right)\le \max_{S^{n-1}} v=v(\bar x).$$
Consequently, we have that
$$ v(\bar x\pm\vartheta)-v(\bar x)\le 0,$$
and this expression vanishes identically for~$\vartheta\in\R^n$
if and only if~$v$ is constant.
Also, by~\eqref{DEFF}, we have that
$$ F'\left( \frac{u(x\pm\vartheta)-u(x)}{|\vartheta|}\right)>0.$$
Therefore, the map
$$ \R^n\ni\vartheta\longmapsto\frac{
F'\left( \frac{u(\bar x+\vartheta)-u(\bar x)}{|\vartheta|}\right)\,\Big( v(\bar x+\vartheta)-v(\bar x)\Big)+
F'\left( \frac{u(\bar x-\vartheta)-u(\bar x)}{|\vartheta|}\right)\,\Big( v(\bar x-\vartheta)-v(\bar x)\Big)
}{|\vartheta|^{n+s+1} }$$
is non-positive, and strictly negative on a positive measure set of~$\vartheta$'s,
unless~$v$ is constant.
This observation and~\eqref{E:4} imply~\eqref{E:5}.

In view of~\eqref{E:5},
we have therefore completed the proof of Theorem~\ref{ADG} under the additional
assumption in~\eqref{REGA}.

We now remove this additional hypothesis. 
To this end, we first make the following general observation:
\begin{equation}\label{OSSE}\begin{split}&
{\mbox{if~$G$ is an $s$-minimal set}}\\&{\mbox{which is a Lipschitz continuous subgraph
in a neighborhood of the origin, with~$0\in\partial G$,}}\\&
{\mbox{then its blow-up cone~$G_0$ is a subgraph
and it is globally Lipschitz continuous.}}
\end{split}\end{equation}
To check this, let us suppose that~$G$, in some~$B_\rho\subset\R^n$, is the subgraph
of a function~$v$ with Lipschitz constant bounded by
some~$M>0$
in~$B_\rho$.
Then~$G_r:=G/r$ is the subgraph of~$v_r(x):=\frac{v(r x)}{r}$
in~$B_{\rho/r}$. Since~$v_r(0)=0$ and the Lipschitz constant of~$v_r$ is bounded 
by~$M$ in~$B_{\rho/r}$, it follows from the Arzel\`a-Ascoli Theorem
that, as~$r\searrow0$,~$v_r$ converges locally uniformly to some~$v_0$,
which has Lipschitz constant bounded by~$M$.
These considerations prove~\eqref{OSSE}.

Now, to remove the additional assumption
in~\eqref{REGA}, we
suppose that there exists~$p\in\partial E$, $p\ne0$,
at which~$E$ is not of class~$C^{3}$.
By Lemma~\ref{0p045:A}, by a blow-up at~$p$,
we obtain a Lipschitz subgraph~$F^\star$ which is an $s$-minimal cone in~$\R^n$
that is singular at the origin.

Suppose first that~$F^\star$ is~$C^3$ outside the origin.
But then we could apply 
Theorem~\ref{ADG} under the additional
assumption in~\eqref{REGA} and conclude that~$F^\star$ is a halfspace, which is a contradiction.

Hence, there must be a boundary point~$q\ne0$ at which~$F^\star$ is not~$C^3$.
But then we can apply again Lemma~\ref{0p045:A} (i.e. we can blow-up $ F^\star$
at the point~$q$),
and find a new $s$-minimal cone in dimension~$n-1$.
Hence, by repeating the blow-up argument at most
a finite number of times, we reach a contradiction, thus completing the proof
of Theorem~\ref{ADG} in its full generality.~\hfill$\Box$

\section{Proof of Theorem~\ref{FV1.1}}

The blow-up cone~$E_0$ constructed before formula~(4.3) in~\cite{FV}
satisfies the assumptions of Theorem~\ref{ADG} and so it is a halfspace.
{F}rom this and~\eqref{down}, Theorem~\ref{FV1.1} plainly follows.~\hfill$\Box$

\section{Proof of Theorem~\ref{FV1.2}}

The blow-down cone~$E_\infty$ constructed before formula~(5.1) in~\cite{FV}
satisfies the assumptions of Theorem~\ref{ADG} and so it is a halfspace.
{F}rom this, Theorem~\ref{FV1.2} plainly follows.~\hfill$\Box$

\section{Proof of Theorem~\ref{MOSBE}}

Without loss of generality, we suppose that~$B=B_1$.
We look at the blow-down limit of~$E$. That is, up to a translation,
we assume that the origin belongs to the boundary of~$E$
and we consider the blow-down cone~$E_\infty$.
We claim that
\begin{equation}\label{M:0}
{\mbox{$E_\infty$ is a Lipschitz subgraph.}}
\end{equation}
For this, we write
\begin{eqnarray*} && E_r=\frac{E}{r}=\left\{ (y,y_{n+1})\in\R^{n+1}
{\mbox{ s.t. }} y=\frac{x}{r},\;\,
y_{n+1}=\frac{x_{n+1}}{r}
{\mbox{ and }} x_{n+1}<u(x)
\right\} \\
&&\qquad\qquad\qquad=\left\{
\left(y,\,\frac{u(ry)}{r}-\mu\right), {\mbox{ with }}
y\in\R^n{\mbox{ and }}\mu>0\right\}.
\end{eqnarray*}
We recall~\eqref{REGMOS}, which also implies that~$u$ is continuous along~$\partial B\subset
\R^n\setminus B$,
and we
take~$M>0$ which controls~$|u|$ along~$\partial B_1$
and the Lipschitz constant of~$u$ in~$\R^n\setminus B_1$.
Then, for any~$x\in\R^n\setminus B_1$,
\begin{equation}\label{980we02}
\begin{split}
& |u(x)|\le \left| u\left(\frac{x}{|x|}\right)\right|+\left| u(x)- u\left(\frac{x}{|x|}\right)\right|
\le M+M\,\left| x- \frac{x}{|x|}\right|\\
&\qquad=M+M\,\left| \frac{x}{|x|}\big( |x|-1\big)\right|
=M+M\,\big| |x|-1\big|=M+M\,\big( |x|-1\big)=M\,|x|.
\end{split}
\end{equation}
Now, we consider the function~$\R^n\ni y\mapsto\frac{u(ry)}{r}=:u_r(y)$
and we observe that~$|u_r(x)|\le M\,|x|$ for any~$x\in\R^n\setminus B_{1/r}$, due to~\eqref{980we02}.
Furthermore, the Lipschitz constant of~$u_r$ in~$x\in\R^n\setminus B_{1/r}$
is bounded by~$M$,
thanks to~\eqref{REGMOS}. {F}rom this
and the Arzel\`a-Ascoli Theorem, up to a subsequence,
we have that~$u_r$ converges locally uniformly to some function~$u_\infty$,
with Lipschitz constant in~$\R^n\setminus\{0\}$ bounded by~$ M$,
and~$E_\infty$ is a subgraph described by the function~$u_\infty$.
This proves~\eqref{M:0}.

Now, from~\eqref{M:0} and Theorem~\ref{ADG}, we infer that~$E_\infty$ is a halfspace.
This and the blow-down theory in~\eqref{down}
imply that~$E$ is a halfspace, as desired.~\hfill$\Box$

\section{Proof of Theorem~\ref{noC}}\label{H1A}

Suppose that~$0\in\partial E$ and let~$E_0$ be the blow-up cone of~$E$.
Since~$E$ is a subgraph, from Lemma~16.3 in~\cite{GIU}, we know that~$E_0$
is a quasi-subgraph, i.e. 
there exists~$u_0:\R^n\to \R\cup\{-\infty\}\cup\{+\infty\}$ such that
$$ E_0=\{x_{n+1}<u(x_1,\dots,x_n)\}.$$
We also stress that the only possible singular point for~$E_0$
is the origin, thanks to
Corollary~\ref{doppio cor} and our assumption on the cones
in one dimension less.

Now, we distinguish two cases: either~$E_0$ is a subgraph, or not.
In the first case, since there are no singular cones in
one dimension less, then~$E_0$ satisfies the assumption of Theorem~\ref{ADG},
and consequently must be a halfspace. Consequently, by~\eqref{up}, it follows that~$E$ is $C^\infty$
near the origin.

Viceversa, if~$E_0$ is not a subgraph, 
there exists~$\tau > 0$ such that~$\partial E_0$ and~$\partial E_0+\tau e_n$
touch at some
point, different than the origin. 
Hence, we can apply the strong maximum principle
at such touching point and conclude that
\begin{equation}
E_0=\tilde E\times \R,
\end{equation}
where~$\tilde E$ is an $s$-minimal cone in one dimension less (see e.g. footnote~3
in~\cite{FV} for the details on the maximum principle).
By assumption, $\tilde E$ has to be a halfspace, which means that~$E_0$
is a (vertical) halfspace.
Once again,
by the theory
of blow-up cones in~\eqref{up}, it follows that~$E$ is a $C^\infty$
near the origin (in vertical coordinates).~\hfill$\Box$

\begin{appendix}

\section{Proof of~\eqref{WP1}}\label{AWP1}

We point out that, since~$F$ is odd
and~\eqref{REGA} holds true, then
for any~$x\ne0$ and~$\vartheta\in B_\varrho(x)$,
with~$\varrho:=\frac12\min\{1, |x|\}$,
\begin{eqnarray*}
&& \frac{
\left|
F\left( \frac{u(x+\vartheta)-u(x)}{|\vartheta|}\right)+
F\left( \frac{u(x-\vartheta)-u(x)}{|\vartheta|}\right)\right|}{|\vartheta|^{n+s} }
=
\frac{
\left|
F\left( \frac{u(x+\vartheta)-u(x)}{|\vartheta|}\right)-
F\left( \frac{u(x)-u(x-\vartheta)}{|\vartheta|}\right)\right|}{|\vartheta|^{n+s} }
\\ &&\qquad\le
\frac{
\const\left|
\left( \frac{u(x+\vartheta)-u(x)}{|\vartheta|}\right)-
\left( \frac{u(x)-u(x-\vartheta)}{|\vartheta|}\right)\right|}{|\vartheta|^{n+s} }=
\frac{
\const\big|
u(x+\vartheta)+u(x-\vartheta)-2u(x)\big|}{|\vartheta|^{n+s+1} }
\le\frac{\const |\vartheta|^{2}}{|\vartheta|^{n+s+1}},
\end{eqnarray*}
which is locally integrable near the origin as a function of~$\vartheta$. This computation
shows~\eqref{WP1}.~\hfill$\Box$

\section{Proof of~\eqref{E:4}}\label{APPP2}

We observe that, for all~$A$, $B$, $a$, $b\in\R$,
since~$F'$ is even, it holds that
\begin{equation}\label{AB:AV} \begin{split}
&
|F'(A)B+F'(a)b|\le \big|F'(A)(B+b)\big|+\big|(F'(A)-F'(a))b\big|\\
&\qquad=\big|F'(A)(B+b)\big|+\big|(F'(A)-F'(-a))b\big|
\le \const \Big( |B+b|+ |A+a|\,|b|\Big).
\end{split}\end{equation}
We take
\begin{eqnarray*}
&& A:=\frac{u(x+\vartheta)-u(x)}{|\vartheta|},\qquad
B:= v(x+\vartheta)-v(x),\\
&&
a:=\frac{u(x-\vartheta)-u(x)}{|\vartheta|},\qquad
b:= v(x-\vartheta)-v(x)
.\end{eqnarray*}
Then, by~\eqref{REGA},
if~$\varrho:=\frac12\min\{1, |x|\}$ and~$\vartheta\in B_\varrho(x)$,
\begin{eqnarray*}
&& |b|\le \const |\vartheta|,\\
&& |A+a|= \frac{|u(x+\vartheta)+u(x-\vartheta)-2u(x)|}{|\vartheta|}\le
\const |\vartheta|\\{\mbox{and }}&&
|B+b|=|v(x+\vartheta)+v(x-\vartheta)-2v(x)|\le\const|\vartheta|^{2}.
\end{eqnarray*}
As a result, for any~$\vartheta\in B_\varrho\subset\R^n$,
\begin{eqnarray*}
&& \frac{
F'\left( \frac{u(x+\vartheta)-u(x)}{|\vartheta|}\right)\,\Big( v(x+\vartheta)-v(x)\Big)+
F'\left( \frac{u(x-\vartheta)-u(x)}{|\vartheta|}\right)\,\Big( v(x-\vartheta)-v(x)\Big)
}{|\vartheta|^{n+s+1} }\le\frac{\const |\vartheta|^{2}}{|\vartheta|^{n+s+1}},
\end{eqnarray*}
which, as a function of~$\vartheta$, belongs to~$L^1(B_\varrho)$.
These considerations give that the integral in~\eqref{E:4}
is well-defined.

Furthermore, for all~$h\in\R$ with~$|h|$ small, if~$\vartheta\in B_\varrho$,
$$ u(x+h e_j\pm\vartheta)-u(x+h e_j)
=u(x\pm\vartheta)-u(x)+h
\int_0^1 v(x+the_j\pm \vartheta)-v(x+the_j)\,dt$$
and therefore
\begin{eqnarray*}
&& F\left( \frac{u(x+h e_j\pm\vartheta)-u(x+h e_j)}{|\vartheta|}\right)
\\ &=&
F\left( \frac{u(x\pm\vartheta)-u(x)}{|\vartheta|}
+\frac{h}{|\vartheta|} 
\int_0^1 v(x+t h e_j\pm \vartheta)-v(x+the_j)\,dt
\right)
\\ &=&
F\left( \frac{u(x\pm\vartheta)-u(x)}{|\vartheta|}\right)\\&&\quad
+\int_0^1
F'\left( \frac{u(x\pm\vartheta)-u(x)}{|\vartheta|}
+\frac{\tau h}{|\vartheta|} 
\int_0^1 v(x+t h e_j\pm \vartheta)-v(x+the_j)\,dt
\right)\,d\tau\\
&&\quad\cdot
\frac{h}{|\vartheta|} 
\int_0^1 v(x+t h e_j\pm \vartheta)-v(x+the_j)\,dt.
\end{eqnarray*}
Hence, we exploit~\eqref{AB:AV} with
\begin{eqnarray*}
&& A:=\frac{u(x+\vartheta)-u(x)}{|\vartheta|}
+\frac{\tau h}{|\vartheta|} 
\int_0^1 v(x+t h e_j+\vartheta)-v(x+the_j)\,dt\\
&&
B:=\frac{1}{|\vartheta|} 
\int_0^1 v(x+t h e_j+ \vartheta)-v(x+the_j)\,dt\\
&& a:=\frac{u(x-\vartheta)-u(x)}{|\vartheta|}
+\frac{\tau h}{|\vartheta|} 
\int_0^1 v(x+t h e_j-\vartheta)-v(x+the_j)\,dt\\
{\mbox{and }}&&
b:=\frac{1}{|\vartheta|} 
\int_0^1 v(x+t h e_j-\vartheta)-v(x+the_j)\,dt.
\end{eqnarray*}
We remark that, in this case,
\begin{eqnarray*}
&&|A+a| \\&\le&
\frac{|u(x+\vartheta)+u(x-\vartheta)-2u(x)|}{|\vartheta|}
+\frac{|h|}{|\vartheta|} 
\int_0^1 \big|v(x+t h e_j+\vartheta)+v(x+t h e_j-\vartheta)-2v(x+the_j)\big|\,dt\\
&\le&\const \min\{1,\,|\vartheta|\}.
\end{eqnarray*}
Also,
$$ |b|\le \const$$
and
\begin{eqnarray*}
|B+b|&\le&
\frac{1}{|\vartheta|} 
\int_0^1\big| v(x+t h e_j+ \vartheta)+ v(x+t h e_j- \vartheta)-2v(x+the_j)\big|\,dt\\
&\le&\const\min\{1,\,|\vartheta|\}.\end{eqnarray*}
In this way, using~\eqref{AB:AV}, we find that
if
$$ {\mathcal{J}}(x,\vartheta):=
\frac{ F\left( \frac{u(x+\vartheta)-u(x)}{|\vartheta|}\right)
+ F\left( \frac{u(x-\vartheta)-u(x)}{|\vartheta|}\right) }{|\vartheta|^{n+s}},$$
then
\begin{eqnarray*}
&&\left|\frac{{\mathcal{J}}(x+he_j,\vartheta)-{\mathcal{J}}(x,\vartheta)}h\right|\\
&=& \frac{1}{|h|\,|\vartheta|^{n+s}}\left|
\int_0^1
F'\left( \frac{u(x+\vartheta)-u(x)}{|\vartheta|}
+\frac{\tau h}{|\vartheta|} 
\int_0^1 v(x+t h e_j+ \vartheta)-v(x+the_j)\,dt
\right)\,d\tau\right.\\
&&\quad\cdot
\frac{h}{|\vartheta|} 
\int_0^1 v(x+t h e_j+ \vartheta)-v(x+the_j)\,dt\\ &&\quad+
F'\left( \frac{u(x-\vartheta)-u(x)}{|\vartheta|}
+\frac{\tau h}{|\vartheta|} 
\int_0^1 v(x+t h e_j- \vartheta)-v(x+the_j)\,dt
\right)\,d\tau\\
&&\quad\cdot\left.
\frac{h}{|\vartheta|} 
\int_0^1 v(x+t h e_j- \vartheta)-v(x+the_j)\,dt\right|\\
&=& \frac{1}{|\vartheta|^{n+s}}\left|\int_0^1
F'(A)B-F'(a)b\,d\tau
\right|\\&\le&\int_0^1
\frac{\const \Big( |B+b|+ |A+a|\,|b|\Big)}{|\vartheta|^{n+s}}\,d\tau\\&\le&
\frac{\const \min\{1,\,|\vartheta|\}}{|\vartheta|^{n+s}}
\end{eqnarray*}
and the latter function belongs to~$ L^1(\R^n)$.
This says that we can make use of the Dominated Convergence Theorem and conclude that
\begin{eqnarray*} &&
\lim_{h\to0} \frac{
{\mathcal{F}}[u](x+he_j)-
{\mathcal{F}}[u](x+h) }{h}
\\&=&\lim_{h\to0} \frac1h\left[
\int_{\R^n}\frac{
F\left( \frac{u(x+he_j+\vartheta)-u(x)}{|\vartheta|}\right)
+F\left( \frac{u(x+he_j-\vartheta)-u(x)}{|\vartheta|}\right)
}{|\vartheta|^{n+s} }-\frac{
F\left( \frac{u(x+\vartheta)-u(x)}{|\vartheta|}\right)+
F\left( \frac{u(x-\vartheta)-u(x)}{|\vartheta|}\right)
}{|\vartheta|^{n+s} }\,d\vartheta
\right]\\
&=&\lim_{h\to0} \int_{\R^n} \frac{
{\mathcal{J}}(x+he_j,\vartheta)- {\mathcal{J}}(x,\vartheta)}{h}\,d\vartheta\\
&=& \int_{\R^n} \lim_{h\to0}\frac{
{\mathcal{J}}(x+he_j,\vartheta)- {\mathcal{J}}(x,\vartheta)}{h}\,d\vartheta\\
&=&  \int_{\R^n} \partial_{x_j} 
{\mathcal{J}}(x,\vartheta)\,d\vartheta\\&=&
\int_{\R^n}\frac{
F'\left( \frac{u(x+\vartheta)-u(x)}{|\vartheta|}\right)\,\Big( v(x+\vartheta)-v(x)\Big)+
F'\left( \frac{u(x-\vartheta)-u(x)}{|\vartheta|}\right)\,\Big( v(x-\vartheta)-v(x)\Big)
}{|\vartheta|^{n+s+1} }\,d\vartheta
.\end{eqnarray*}
This establishes~\eqref{E:4}, as desired.~\hfill$\Box$

\end{appendix}

\vfill


\begin{thebibliography}{1}

\bibitem{AB} 
{\sc N. Abatangelo, E. Valdinoci},
{\em A notion of nonlocal curvature}.
Numer. Funct. Anal. Optim. 35 (2014), no. 7-9, 793--815. 

\bibitem{barrios} 
{\sc B. Barrios, A. Figalli, E. Valdinoci}, 
{\em Bootstrap regularity for integro-differential operators and its application to
nonlocal minimal surfaces}.
Ann. Sc. Norm. Super. Pisa Cl. Sci. (5) 13 (2014), no. 3, 609--639.

\bibitem{CFW}
{\sc X. Cabr\'e, M. M. Fall, J. Sol\`a-Morales, T. Weth},
{\em Curves and surfaces
with constant nonlocal mean curvature: meeting Alexandrov and Delaunay}.
In print on J. Reine Angew. Math.
{\tt https://doi.org/10.1515/crelle-2015-0117}

\bibitem{dan}
{\sc L. Caffarelli, D. De Silva, O. Savin}, 
{\em Obstacle-type problems for minimal surfaces}. 
Comm. Partial Differential Equations 41 (2016), no. 8, 1303--1323.

\bibitem{CRS}
{\sc L. Caffarelli, J.-M. Roquejoffre, O. Savin},
{\em Nonlocal minimal surfaces}.
Comm. Pure Appl. Math. 63 (2010), no. 9, 1111--1144.

\bibitem{takis}
{\sc L. A. Caffarelli, P. E. Souganidis},
{\em Convergence of nonlocal threshold dynamics approximations to
front propagation}.
Arch. Ration. Mech. Anal. 195 (2010), no. 1, 1--23.

\bibitem{CV}
{\sc L. Caffarelli, E. Valdinoci},
{\em Regularity properties of nonlocal minimal surfaces via limiting arguments}.
Adv. Math. 248 (2013), 843--871. 

\bibitem{CMP15}
{\sc A. Chambolle, M. Morini, M. Ponsiglione},
{\em Nonlocal curvature flows}.
Arch. Ration. Mech. Anal. 218 (2015), no. 3, 1263--1329.

\bibitem{sinestrari}
{\sc E. Cinti, C. Sinestrari, E. Valdinoci},
{\em Neckpinch singularities in fractional mean curvature flows}.
{\tt https://arxiv.org/abs/1607.08032}

\bibitem{ciraolo}
{\sc G. Ciraolo, A. Figalli, F. Maggi, M. Novaga},
{\em Rigidity and sharp stability estimates for
hypersurfaces with constant and almost-constant
nonlocal mean curvature}.
In print on J. Reine Angew. Math.
{\tt https://doi.org/10.1515/crelle-2015-0088}

\bibitem{CIME}
{\sc M. Cozzi, A. Figalli},
{\em Regularity theory for local and nonlocal minimal surfaces: an overview}.
Lecture Notes in Math., Fond. CIME/CIME Found. Subser., Springer, Cham.

\bibitem{NA}
{\sc J. D\'avila, M. del Pino, S. Dipierro, E. Valdinoci},
{\em Nonlocal Delaunay surfaces}. Nonlinear Anal. 137 (2016), 357--380.

\bibitem{DDW}
{\sc J. D\'avila, M. del Pino, J. Wei},
{\em Nonlocal $s$-minimal surfaces and Lawson cones}.
In print on
J. Differential Geom.,
{\tt https://arxiv.org/abs/1402.4173}

\bibitem{DG1} 
{\sc E. De Giorgi},
{\em Una estensione del teorema di Bernstein}.
Ann. Scuola Norm. Sup. Pisa (3) 19 1965 79--85.

\bibitem{DG2} 
{\sc E. De Giorgi},
{\em Errata-Corrige: ``Una estensione del teorema di Bernstein''}.
Ann. Scuola Norm. Sup. Pisa Cl. Sci. (3) 19 (1965), no. 3, 463--463.

\bibitem{DSV16}
{\sc S. Dipierro, O. Savin, E. Valdinoci},
{\em Graph properties for nonlocal minimal surfaces}.
Calc. Var. Partial Differential Equations 55 (2016), no. 4, Paper No. 86, 25 pp. 

\bibitem{DSV17}
{\sc S. Dipierro, O. Savin, E. Valdinoci},
{\em Boundary behavior of nonlocal minimal surfaces}.
J. Funct. Anal. 272 (2017), no. 5, 1791--1851.

\bibitem{survey}
{\sc S. Dipierro, E. Valdinoci},
{\em Nonlocal minimal surfaces: interior regularity,
quantitative estimates and boundary stickiness}. Recent
Dev. Nonlocal Theory, De Gruyter, Berlin.

\bibitem{I5}
{\sc A. Figalli, N. Fusco, F. Maggi, V. Millot, M. Morini}, 
{\em Isoperimetry and stability properties of balls 
with respect to nonlocal energies}.
Comm. Math. Phys. 336 (2015), no. 1, 441--507.

\bibitem{FV}
{\sc A. Figalli, E. Valdinoci},
{\em Regularity and Bernstein-type results for nonlocal minimal surfaces}.
In print on  J. Reine Angew. Math.,
{\tt https://doi.org/10.1515/crelle-2015-0006}

\bibitem{FSISO}
{\sc R. L. Frank, R. Seiringer}, 
{\em Non-linear ground state representations and sharp Hardy inequalities}. 
J. Funct. Anal. 255 (2008), no. 12, 3407--3430. 

\bibitem{FMM}
{\sc N. Fusco, V. Millot, M. Morini}, 
{\em A quantitative isoperimetric inequality for fractional perimeters}. 
J. Funct. Anal. 261 (2011), no. 3, 697--715.

\bibitem{GR}
{\sc A. M. Garsia, E. Rodemich},
{\em Monotonicity of certain functionals under rearrangement}.
Colloque International sur les Processus
Gaussiens et les Distributions Al\'eatoires
(Colloque Internat. du CNRS, No. 222, Strasbourg, 1973).
Ann. Inst. Fourier (Grenoble) 24 (1974), no. 2, vi, 67--116.

\bibitem{GIU}
{\sc E. Giusti}, {\em Minimal surfaces and functions of bounded variation}.
Monographs in Mathematics, 80. Birkh\"auser Verlag, Basel, 1984. xii+240 pp. 

\bibitem{I}
{\sc C. Imbert},
{\em Level set approach for fractional mean curvature flows}.
Interfaces Free Bound. 11 (2009), no. 1, 153--176.

\bibitem{LOMB}
{\sc L. Lombardini},
{\em Fractional Perimeter and Nonlocal Minimal Surfaces}.
Master's Thesis, Universit\`a di Milano (2015), arXiv:1508.06241.

\bibitem{MOS}
{\sc J. Moser},
{\em On Harnack's theorem for elliptic differential equations}. 
Comm. Pure Appl. Math. 14 1961 577–591.

\bibitem{saez}
{\sc M. S\'aez, E. Valdinoci},
{\em On the evolution by fractional mean curvature}.
In print on Comm. Anal. Geom.

\bibitem{SV2}
{\sc O. Savin, E. Valdinoci},
{\em $\Gamma$-convergence for nonlocal phase transitions}.
Ann. Inst. H. Poincar\'e Anal. Non Lin\'eaire 29 (2012), no. 4, 479--500.

\bibitem{SV}
{\sc O. Savin, E. Valdinoci},
{\em Regularity of nonlocal minimal cones in dimension $2$}. 
Calc. Var. Partial Differential Equations 48 (2013), no. 1-2, 33--39.

\end{thebibliography}
\end{document}